\documentclass[12pt]{amsart}
\usepackage[all]{xy}

\usepackage{graphicx}  

\usepackage{amsmath}
\usepackage{amssymb}
\usepackage{amsfonts}
\usepackage{mathrsfs}
\usepackage{mathtools}
\usepackage[all]{xy}
\usepackage{color}

\newcommand{\B}{\mbox{${\mathcal B}$}}
\newcommand{\C}{\mbox{${\mathcal C}$}}
\newcommand{\D}{\mbox{${\mathcal D}$}}

\newcommand{\F}{\mbox{${\mathcal F}$}}

\renewcommand{\H}{\mbox{${\mathcal H}$}}

\newcommand{\J}{\mbox{${\mathcal J}$}}
\newcommand{\K}{\mbox{${\mathcal K}$}}
\newcommand{\Ll}{\mbox{${\mathcal L}$}}
\newcommand{\M}{\mbox{${\mathcal M}$}}
\newcommand{\N}{\mbox{${\mathcal N}$}}

\newcommand{\T}{\mbox{${\mathcal T}$}}

\newcommand{\V}{\mbox{${\mathcal V}$}}

\newcommand{\X}{\mbox{${\mathcal X}$}}

\newcommand{\fecho}[1]{\overline{#1}}
\newcommand{\intern}[1]{\langle #1\rangle}
\newcommand{\sspan}{\operatorname{span}}

\newtheorem{theorem}{Theorem}[section]
\newtheorem{lemma}[theorem]{Lemma}
\newtheorem{corollary}[theorem]{Corollary}
\newtheorem{proposition}[theorem]{Proposition}

\theoremstyle{remark}
\newtheorem{remark}[theorem]{\bf Remark}
\theoremstyle{definition}
\newtheorem{example}[theorem]{\bf Example}
\numberwithin{equation}{section}

\def\spn{\textup{span}}
\begin{document}

\title[Lie modules of nest algebras]{Weakly closed  Lie modules of nest algebras}

\author[L. Oliveira]{Lina Oliveira}
\address{Center for Mathematical Analysis,
Geometry and Dynamical Systems {\sl{and}} Department of Mathematics\\
Instituto Superior T\'{e}cnico, Universidade de Lisboa
\\ Av. Rovisco Pais
1049-001 Lisboa Portugal}
\email{linaoliv@math.tecnico.ulisboa.pt}

\author[M. Santos]{Miguel Santos}
\address{Instituto Superior T\'{e}cnico, Universidade de Lisboa
\\ Av. Rovisco Pais
1049-001 Lisboa Portugal}
\email{miguel.m.santos@ist.utl.pt }
\date{December 10, 2015}

\keywords{Bimodule, Lie ideal, Lie module, nest algebra}
\thanks{The first author was partially funded by FCT/Portugal through UID/MAT/04459/2013 and EXCL/MAT-GEO/0222/2012.  The second author was partially supported by a fellowship of the program ``Novos Talentos em Matem\'{a}tica'' of the Gulbenkian Foundation.}

\subjclass[2010]{47L35, 46K50, 17B60}

\begin{abstract}
Let $\T(\N)$ be a nest algebra of  operators on Hilbert space and let $\Ll$ be a weakly closed Lie $\T(\N)$-module. 
We construct explicitly the largest possible weakly closed $\T(\N)$-bimodule $\J(\Ll)$ and a weakly closed $\T(\N)$-bimodule $\K(\Ll)$
such  that 
$$
\J(\Ll)\subseteq \Ll \subseteq  \K(\Ll) +\D_{\K(\Ll)},
$$
 $[\K(\Ll), \T(\N)]\subseteq \Ll$ and  $\D_{\K(\Ll)}$ is a von Neumann subalgebra of the diagonal $\T(\N)\cap \T(\N)^*$. 
\end{abstract}

\maketitle

\section{Introduction}\label{prelim}
It has been established in  \cite{HMS} 
 that  any weakly closed Lie ideal $\Ll$ of a nest algebra $\T(\N)$ of operators on Hilbert space contains a weakly closed associative ideal of $\T(\N)$ and   is contained in a sum of this ideal with a von Neumann  subalgebra of the diagonal  $\D(\N)$ of the nest algebra. That is to say that there exist a weakly closed associative ideal $\K(\Ll)$ and a
von Neumann subalgebra
$\D_{\K(\Ll)}$ of $\D(\N)$ such that
\begin{equation}\label{00}
\K(\Ll)\subseteq \Ll\subseteq \K(\Ll)+\D_{\K(\Ll)}.
\end{equation}

The purpose of the present  work is to show that a similar result holds when we pass from ideals to modules. More precisely,  the main result Theorem \ref{t_sum} asserts that, if 
 $\Ll$ is a weakly closed Lie $\T(\N)$-module, then
 \begin{equation}\label{01}
\J(\Ll)\subseteq \Ll \subseteq  \K(\Ll) +\D_{\K(\Ll)},
\end{equation}
 where  $\D_{\K(\Ll)}$ is a von Neumman subalgebra  of the diagonal $\D(\N)$, 
 $$\J(\Ll) = \fecho{\sspan}^w\bigl(\{Q\B(\H) P^\perp  \colon  P,Q\in \N, Q\B(\H) P^\perp \subseteq \Ll\}\bigr)$$
  is the largest weakly closed  $\T(\N)$-bimodule contained in $\Ll$ and $\K(\Ll)$ is a weakly closed  $\T(\N)$-bimodule such that  $[\K(\Ll),\T(\N)]\subseteq \Ll$, a result reminiscent of \cite{FMS}, Theorem 2. 
  
  Neither is it  necessarily the case that $\J(\Ll)$ be   a subset of $\K(\Ll)$  nor  that $\Ll$ be contained in  $\K(\Ll)$, as Example \ref{example} shows.
 However, when $\Ll$ is in fact a weakly closed Lie ideal,  a refinement of both \eqref{00} and \eqref{01} can be obtained, as is outlined in Remark \ref{finalrem}. In this situation, \eqref{00} and \eqref{01} coalesce yielding
  \begin{equation}\label{02}
\K(\Ll)\subseteq \J(\Ll)\subseteq \Ll \subseteq  \K(\Ll) +\D_{\K(\Ll)},
\end{equation}
and $\K(\Ll)$ might even be a proper subset of $\J(\Ll)$.

 \medskip
 This work is organised in three sections.  Section \ref{bimod} addresses the question of describing the weakly closed $\T(\N)$-bimodules, and this is done in an essentially different way from those of \cite{DDH, EP}. The main result of Section \ref{bimod} is Lemma \ref{l_eqcond} which gives several characterisations of  weakly closed  $\T(\N)$-bimodules. This lemma is then used to prove Corollary \ref{c_largebim}, which provides a constructive description of the largest  weakly closed  $\T(\N)$-bimodule contained in any given weakly closed  subspace of $\T(\N)$.
 Section \ref{liemod} is devoted to the proof of Theorem \ref{t_sum}.

The notation is set  in this final part of  Section \ref{prelim}  and   some facts needed in the sequel are also recalled.

\medskip

Let $\H$ be a complex Hilbert space, let $\B(\H)$ be the complex Banach space of bounded linear operators on $\H$ and let $\F_1(\H)$ be  the set of rank one operators in $\B(\H)$.
A totally ordered family $\N  $ of projections in $\B(\H)$ containing $0 $ and the identity $I $
is said to be a \textit{nest}. If, furthermore, $\N  $ is a complete sublattice
of the lattice of projections in $\B(\H)$,  then $\N $ is called a \emph{complete 
nest}. 
 The \emph{nest 
algebra} $\T(\N)$ associated 
with a nest $\N  $ is the subalgebra of all operators $T $ in $B  (\H )     $ such that, for all projections $P $ 
in $\N  $, 
 $T(P  (\H ))
\subseteq P  (\H ), 
$
or, equivalently,
an operator $T $ in $B  (\H ) $ lies in $\T(\N)  $ if and only 
if, for all projections $P $ in the nest $\N  $,
$   
P^\perp   T P = 0
$, where $P^\perp = I-P$. Each nest is 
contained in a complete nest which generates the same nest algebra 
(cf. \cite{kd, ring}). Henceforth only complete nests will be considered. 

The algebra $\T(\N)  $ is a  
weakly closed subalgebra of
$B  (\H ) $, the \emph{diagonal} $\D(\N)$ of which   is the von Neumann algebra defined by $\D(\N)=\T(\N)\cap \T(\N)^*$.

A nest algebra $\T(\N)$ together with the product defined, for all operators $T$ and $S$ in $\T(\N)$, by 
$
[T,S]=TS-ST
$ 
is a Lie algebra. A  complex subspace $\M$ of  $\B(\H) $
is said to be a \emph{$\T(\N)$-bimodule} if $\M\T(\N),\T(\N)\M \subseteq \M$ and is called a \emph{Lie $\T(\N)$-module}  if $[\M, \T(\N)]\subseteq \M$. Lie $\T(\N)$-modules and  $\T(\N)$-bimodules contained in the nest algebra $\T(\N)$ are called, respectively, Lie ideals and ideals of $\T(\N)$.
In the sequel,  Lie $\T(\N)$-modules may be referred to as Lie modules for simplicity. For the same reason,  $\T(\N)$-bimodules may  be called  simply bimodules.

Let  $x$ and $y$ be elements of the Hilbert space  $\H$ and  let $x\otimes y$ be the rank one operator   defined, for all $z$ in $\H$, by 
 $
  z\mapsto \intern{z, x}y
  $,
   where $ \intern{\cdot, \cdot}$ denotes the inner product of $\H$.   Let $P$ be a projection in the nest $\N$ and  let $P_{-}$  be the projection in $\N$ defined by  $P_{-} =\vee \{Q\in \N: Q<P\}$. 
 A rank one operator $x\otimes y$ lies in $\T(\N)$ if, and only if,  there exists a projection $P$ such that 
   $P_- x=0$
    and 
  $ Py=y$; moreover, $P$ can be chosen to be equal to  $\bigwedge\{Q\in \N: Qy=y\}$ (cf. \cite{ring}). 
For  the general theory  of nest algebras, the reader is referred to \cite{kd, ring}.

In what follows, the closure  in the norm topology of the span  of a subset  $\X$ of $\B(\H)$  will be  denoted by $\fecho{\sspan}(\X)$, whereas its closure in the weak operator topology will be  denoted by  $\fecho{\sspan}^w(\X)$.  All subspaces either of $\H$ or  of $\B(\H)$ are assumed to be complex subspaces. 

\section{Bimodules}\label{bimod}
Let $z$ be an element of the Hilbert space $\H$ and   let  $P_z$ and $\hat{P}_z$ be the  projections  defined by
 \[
 P_z=\wedge\{Q\in \N: Qz=z\}\quad  \text{,} \quad \quad \hat{P}_z=\vee \{Q\in \N: Qz=0\}.
 \]
 The projections $P_z$ and $\hat{P}_z$ lie in the nest $\N$ and 
 $
 P_zz=z$,  $\hat{P}_z z=0$.
Following \cite{LO2},  each rank one operator $x\otimes y$ will be associated with the projections $\hat{P}_x$ and $P_y$. 
\begin{remark}\label{rem01}
Observe that, for all $x,y\in \H$ and all $T\in \T(\N)$,  
 $P_{Ty}\leq P_y$ and  $\hat{P}_{x}\leq \hat{P}_{T^*x}$. In fact,
 $
Ty =P_yTP_yy=P_yTy
$
 which shows that $P_{Ty}\leq P_y$. Similarly  
$,
T^*x=\hat{P}_{x}^\perp  T^*\hat{P}_{x}^\perp  x
=\hat{P}_{x}^\perp  T^* x
$,
yielding that $\hat{P}_{x}\leq \hat{P}_{T^*x}$.
\end{remark}

\begin{lemma}\label{l_base}
Let $\V$ be a weakly closed $T(\N)$-bimodule and let $x\otimes y$ be a rank one operator  in $\B(\H)$. Then $x\otimes y$ lies in $\V$ if and only if $P_y \B(\H) \hat{P}_x^\perp \subseteq \V$.
\end{lemma}

\begin{proof}
Let $x\otimes y$ be a rank one operator lying in $\V$. Then, by \cite{EP}, Theorem 1.5, there exists a left order continuous homomorphism $P\mapsto \tilde{P}$ on $\N$ such that an operator $T\in \B(\H)$ lies in $\V$ if and only if, for all $P\in \N$, 
$
 \tilde{P}^\perp  TP=0
$. 
Let $T$ lie in $P_yB(\H) \hat{P}_x^\perp$ and suppose that $P\in \N$ is a projection such that  $P\leq  \hat{P}_x$. Then, 
\begin{equation}\label{07}
 \tilde{P}^\perp  TP= \tilde{P}^\perp  P_yT\hat{P}_x^\perp  P=0,
\end{equation}
which shows that $T\in \V$.
 Suppose now that $\hat{P}_x<P$.   Since $x\otimes y\in \V$, by the definition of $\tilde{P}$,   $P_y\leq \tilde{P}$ (see \cite{EP}, p. 221). Hence
\begin{equation}\label{08}
\tilde{P}^\perp  TP= \tilde{P}^\perp  P_yT\hat{P}_x^\perp  P=0.
\end{equation}
Combining \eqref{07}-\eqref{08} yields that $P_yB(\H) \hat{P}_x^\perp \subseteq \V$. 

The converse assertion is clear.
\end{proof}

The underlying role of the rank one operators in the construction of a bimodule, outlined by Lemma~\ref{l_base}, suggests that the following definitions be made.  Let $\M$ be a   subspace of $\B(\H)$ and let  $\C(\M)$ 
   be the subset of $\B(\H)$  defined  by 
\begin{equation}\label{03}
\C(\M) =\{x\otimes y\in \B(\H)\colon P_y\B(\H) \hat{P}_x^\perp  \subseteq \M\},
\end{equation}
and let  
\begin{equation}\label{eq6}
\J(\M)=\fecho{\sspan}^w (\C(\M))
\end{equation}
 Since $x\otimes y=P_y(x\otimes y) \hat{P}_x^\perp $, for any given rank one operator $x\otimes y$, it immediately follows that $\C(\M)\subseteq \M$ and, consequently, it is also the case that $\sspan( \C(\M))\subseteq \M$.

\begin{lemma}\label{l_uem}
Let $\M$ be  a subspace of $\B(\H)$. Then $\sspan(\C(\M))$  is a 
 $\T(\N)$-bimodule contained in $\M$, and 
 $\J(\M)$ is a weakly closed $\T(\N)$-bimodule such that 
$$
\J(\M) =\fecho{\sspan}^w (\C(\M))= \fecho{\sspan}^w\bigl(\bigcup_{ x\otimes y\in \C(\M)}  P_y\B(\H) \hat{P}_x^\perp \bigr).
$$
   \end{lemma}

 \begin{proof}
To prove that $\sspan(\C(\M))$  is a 
 $\T(\N)$-bimodule contained in $\M$, it suffices to show that, for all $T\in \T(\N)$,    $T\C(\M), \C(\M)T\subseteq \C(\M)$.
 Suppose that $x\otimes y$ is an operator in  $\C(\M)$ and consider the rank one operator $T(x\otimes y)=x\otimes Ty$. Then, since $P_{Ty}\leq P_y$ (see Remark \ref{rem01}) and $P_{Ty}\B(\H) \subseteq \B(\H)$, it follows that
\[
P_{Ty}\B(\H) \hat{P}_x^\perp =P_yP_{Ty}\B(\H) \hat{P}_x^\perp \subseteq P_y\B(\H) \hat{P}_x^\perp \subseteq \M.
\]
Hence, $T(x\otimes y)$ lies in $\C(\M)$.
Similarly, considering the product $(x\otimes y)T=T^*x\otimes y$, 
\[
P_{y}\B(\H) \hat{P}_{T^*x}^\perp =P_y\B(\H) \hat{P}_{T^*x}^\perp  \hat{P}_x^\perp \subseteq P_y\B(\H) \hat{P}_x^\perp \subseteq \M,
\]
which shows that $(x\otimes y)T$ lies in $\C(\M)$.
Consequently,  both $\T(\N)\C(\M)$ and $\C(\M) \T(\N)$ are subsets of  $\C(\M)$ and, therefore, $\sspan( \C(\M))$ is a bimodule.  
 Since the multiplication in $\B(\H)$ is  separately continuous  in the weak operator topology, it follows that the closure of a bimodule in this topology is also a bimodule. Hence,  $\J(\M)$ is a weakly closed bimodule.

It is clear that any operator $x\otimes y$  in $\C(\M)$ must also lie $\J(\M)$. Hence, by Lemma \ref{l_base}, $P_y\B(\H) \hat{P}_x^\perp\subseteq \J(\M)$. Consequently, 
$
\bigcup_{ x\otimes y\in \C(\M)}  P_y\B(\H) \hat{P}_x^\perp \subseteq \J(\M)
$
and, therefore, 

\begin{equation}\label{eq01}
\fecho{\sspan}^w\bigl(\bigcup_{ x\otimes y\in \C(\M)}  P_y\B(\H) \hat{P}_x^\perp \bigr)\subseteq \J(\M).
\end{equation}

\noindent
On the other hand, since $x\otimes y\in P_y\B(\H) \hat{P}_x^\perp$ for any given operator $x\otimes y$, it follows that 
$$
\sspan(\C(\M))\subseteq \sspan \bigl(\bigcup_{ x\otimes y\in \C(\M)}  P_y\B(\H) \hat{P}_x^\perp\bigr).
$$
Hence
\begin{equation}\label{eq02}
\J(\M)=\fecho{\sspan}^w (\C(\M))\subseteq \fecho{\sspan}^w \bigl(\bigcup_{ x\otimes y\in \C(\M)}  P_y\B(\H) \hat{P}_x^\perp\bigr).
\end{equation}

Combining (\ref{eq01})-(\ref{eq02}) yields that 
$$
\J(\M) = \fecho{\sspan}^w\bigl(\bigcup_{ x\otimes y\in \C(\M)}  P_y\B(\H) \hat{P}_x^\perp \bigr),
$$
as required.
\end{proof}

\begin{remark}\label{rem02}
It is obvious from Lemma~\ref{l_uem}  that $
\fecho{\sspan}(\C(\M))
$   is the smallest norm closed  bimodule containing $\C(\M)$. It is also clear that   $
\J(\M)
$
 is the smallest   weakly closed  bimodule containing $\C(\M)$.
\end{remark}

\begin{proposition}\label{p_eqcond} Let $\M$ be a  subspace of $\B(\H)$ and let $\J(\M)$ be as in \eqref{eq6}. Then  $\M$ is a weakly closed $\T(\N)$-bimodule if and only if 
$
{\M} =\J(\M).
$
\end{proposition}

\begin{proof}
It immediately follows from   Lemma \ref{l_uem}  that,  if $
{\M} =\J(\M),
$ 
then $\M$ is a weakly closed bimodule. 
Conversely, suppose  that $\M$ is a weakly closed bimodule. Then, by Remark \ref{rem02},  $\J(\M)$ is the smallest weakly closed bimodule containing $\M$ and, therefore, coincides  with  $\M$ itself.
\end{proof}

\begin{corollary}\label{c_rankone} Let $\M_1$ and $\M_2$ be weakly closed $\T(\N)$-bimodules such that 
$\M_1\cap \F_1(\H)=\M_2\cap \F_1(\H)$. Then $\M_1$ and $\M_2$ coincide.
\end{corollary}

\begin{proof}
It is an immediate  consequence of  Lemma \ref{l_base} and  Proposition \ref{p_eqcond} 
 that weakly closed bimodules having the same subset of rank one operators must coincide.
\end{proof}

Although the   map $\phi$ in the next lemma be  generally presented,   its definition is in fact rooted   in the investigation of  the structure of weakly closed Lie $\T(\N)$-modules initiated in \cite{AO}.
\begin{lemma}\label{l_phi}Let $\M$ be a weakly closed subspace of $\B(\H)$. Then the map $P\mapsto \phi(P)$   defined on $\N$ by 
\begin{equation}\label{eq11}
\phi(P)=\vee\{Q\in \N \colon Q \B(\H)R^\perp  \subseteq \M, R\in \N, R< P\} 
\end{equation}
is a  left order continuous homomorphism on $\N$. 
Moreover, 
for all  $P\in \N$,  
\begin{equation}\label{eq12}
\phi(P)\B(\H)P^\perp \subseteq \phi(P)\B(\H)P_{-}^\perp\subseteq \M.
\end{equation}
\end{lemma}

\begin{proof}
If $P_1\leq P_2$, then 
$
\{Q\in \N \colon Q \B(\H)R^\perp  \subseteq \M, R\in \N, R< P_1\}$ is a subset of  $\{Q\in \N \colon Q \B(\H)R^\perp  \subseteq \M, R\in \N, R< P_2\} 
$,
from which immediately follows that $\phi(P_1)\leq \phi(P_2)$. Hence $\phi$  is  an order  homomorphism on $\N$.
Notice that, for all projections $Q\in \N$,
\begin{equation}\label{eq13}
Q\B(\H)P_2^\perp\subseteq Q\B(\H){P_1}^\perp.
\end{equation}
It follows from \eqref{eq13} that, for all $Q,P\in \N$,
\begin{equation}\label{eq17}
 Q\B(\H)P^\perp\subseteq {Q\B(\H)(P}_{-})^\perp.
\end{equation}

Let $(Q_l)$ be a net  in $\{Q\in \N \colon Q \B(\H)R^\perp  \subseteq \M, R\in \N, R< P\} 
$ 
 strongly converging to $\phi(P)$. For all $l$, there exists  a projection $R_l$  such that $R_l<P$ and  
$$
Q_l\B(\H)P^\perp\subseteq  Q_l\B(\H)R_l^\perp \subseteq \M.
$$
Consequently, if $P=P_{-}$, then 
$$
Q_l\B(\H)P^\perp=Q_l\B(\H)P_{-}^\perp \subseteq \M.
$$
Since $\M$ is weakly closed and, hence, strongly closed, by the separate  strong continuity of the multiplication, it now follows that 
$
\phi(P)\B(\H)P^\perp =\phi(P)\B(\H)P_{-}^\perp \subseteq \M.
$

If, on the other hand,  $P_{-}<P$, then, combining  \eqref{eq11} and \eqref{eq13},   
\[
\phi(P)=\vee\{Q\in \N \colon Q \B(\H)P_{-}^\perp  \subseteq \M\}.
\] Consequently,  there exists a net $(Q_j)$ strongly  converging to $\phi(P)$ and such that, for all $j$, 
$Q_j \B(\H)P_{-}^\perp  \subseteq \M$. Observe that it is also the case that $(Q_j)$ weakly  converges to $\phi(P)$. 
It now follows from the 
 separate continuity of the multiplication in the weak operator topology that $\phi(P)\B(\H)P_{-}^\perp\subseteq \fecho{\M}^w=\M$.
 Hence, using \eqref{eq17},
$$
\phi(P)\B(\H)P^\perp \subseteq \phi(P)\B(\H)P_{-}^\perp \subseteq \M,
$$
 which concludes the proof of  \eqref{eq12}.
 
 It only remains to show that  the map $\phi$ is   left order continuous; that is to say that, for every subset $\X$ of $\N$, $\phi(\vee \X)=\vee \phi(\X)$.    This trivially holds for the empty set. Suppose then in what follows that $\X \neq \emptyset$.
 If $\vee \X\in \X$, then the equality $\phi(\vee \X)=\vee \phi(\X)$ is obvious, since $\phi$ is an order-preserving map. If, on the other hand, $\vee \X\notin \X$ then $(\vee \X)_{-}=\vee \X$.

Hence, suppose now that  $P \in\N$ is such that $P_{-}=P$. In this case,  $P=\vee\{R\in \N\colon R<P\}$ and, since $\phi$ is an order homomorphism, it is clear that $\vee\{\phi(R)\in \N\colon R<P\}\leq\phi(P)$.

If $\vee\{\phi(R)\in \N\colon R<P\}<\phi(P)$, then by \eqref{eq11}  there would exist projections $R', Q\in \N$ such that $R'<P$, $\vee\{\phi(R)\in \N\colon R<P\}<Q$ and 
$Q \B(\H)R'^\perp  \subseteq \M$. But, in this case,
$Q\leq \vee\{\phi(R)\in \N\colon R<P\}$ yielding a contradiction.
 It follows that  
$\phi(P)=\vee\{\phi(R)\in \N\colon R<P\}$, as required. 
 Letting  $P=\vee\X$, one finally has $\phi(\vee \X)=\vee \phi(\X)$, which concludes the proof.
 \end{proof}

\begin{lemma}\label{l_eqcond} Let $\M$ be a  subspace of $\B(\H)$ and let $\J(\M)$ be as in \eqref{eq6}. The following assertions are equivalent.
\begin{enumerate}
\item[(i)] $\M$ is a weakly closed $\T(\N)$-bimodule.
\item[(ii)] $
{\M} =\J(\M).
$

\item[(iii)] There exists a  left order continuous homomorphism $\phi\colon \N\to \N$,
 defined  by  
$$
\phi(P)=\vee\{Q\in \N \colon Q \B(\H)R^\perp  \subseteq \M, R\in \N, R< P\},
$$ 
and such that 
\begin{equation}\label{04}
\M=\{T\in \B(\H)\colon \phi(P)^\perp  TP=0\}.
\end{equation}

\item[(iv)] $$
\M=\fecho{\sspan}^w\bigl(\{Q\B(\H) P^\perp  \colon  P,Q\in \N, Q\B(\H) P^\perp \subseteq \M\}\bigr)
.$$
\end{enumerate}
\end{lemma}

\begin{remark}\label{rem03} 
 It is  clear  that  a  set defined by any (not necessarily  order-preserving) map $\phi$ as in  \eqref{04} is  a weakly closed $\T(\N)$-bimodule. It has already been shown in  \cite{EP} that, for each weakly closed $\T(\N)$-bimodule $\M$, there exists a (not necessarily unique)  left order continuous map $P\mapsto \tilde{P}$  on $\N$  describing $\M$ in the sense of \eqref{04}. However,  the map $P\mapsto \tilde{P}$  and the map $\phi$ above are differently defined and, whilst the  proof of the existence of the map $P\mapsto \tilde{P}$ relies on the decomposability of the finite rank operators,  the proof below is  free of that assumption.
\end{remark}

\begin{proof} 
The equivalence (i)  $\Leftrightarrow$ (ii) has already been proved in Proposition \ref{p_eqcond}. 

As observed in Remark \ref{rem03}, the  set defined by  \eqref{04} is necessarily a weakly closed $\T(\N)$-bimodule. Hence (iii) $\Rightarrow$ (i)  $\Rightarrow$ (ii). 

To  show that (ii)  $\Rightarrow$ (iii), assume now that $\M=\J(\M)$.  It will be shown  firstly that   $
\C(\M)$ coincides with the subset of rank one operators contained in $\{T\in \B(\H)\colon \phi(P)^\perp  TP=0\}$.
  Let $x\otimes y$ be a rank one operator in $\C(\M)$, let $P$ be a projection in $\N$ and suppose initially  that $\hat{P}_x \leq P$. It follows, by Lemma \ref{l_base} and the definition of $\phi$, that $P_y\leq \phi(P)$ and, consequently, 
$$
\phi(P)^\perp (x\otimes y)P=\phi(P)^\perp  P_y(x\otimes y)\hat{P}_x^\perp  P=0.
$$
It can  be similarly shown  that $\phi(P)^\perp (x\otimes y)P=0$, when  $P < \hat{P}_x $. Hence 
$
\C(\M)\subseteq \{T\in \B(\H)\colon \phi(P)^\perp  TP=0\}
$.

Conversely, let  $x\otimes y$ be an operator lying in the weakly closed bimodule 
$$
\{T\in \B(\H)\colon \phi(P)^\perp  TP=0\}.
$$
 It will be
shown that $x\otimes y$ lies in $\C(\M)$. By Lemma \ref{l_base}, this is equivalent to proving that, for all $T\in P_y\B(\H)\hat{P}_x^\perp$ and  all $P\in \N$, 
$
\phi(P)^\perp TP=0
$.
 Suppose then that $T$ is an operator in $P_y\B(\H)\hat{P}_x^\perp$. If $P\leq \hat{P}_x$, then
  $$
  \phi(P)^\perp  TP=\phi(P)^\perp P_yT\hat{P}_x^\perp P=0.
  $$
On the other hand, if   $\hat{P}_x< P$ then, since $x\otimes y$ lies in $\{T\in \B(\H)\colon \phi(P)^\perp  TP=0\}$,  it  immediately follows from the definition of $\phi$  that $P_y\leq \phi(P)$.  Hence 
$$
  \phi(P)^\perp  TP=\phi(P)^\perp P_yT\hat{P}_x^\perp P=0.
  $$

It has been shown that a rank one operator lies in the weakly closed bimodule $\{T\in \B(\H)\colon \phi(P)^\perp  TP=0\}$ if and only if it lies in $\C(\M)$ or, in other words, if the rank one operators in both sets are exactly the same.  
By Corollary \ref{c_rankone},  weakly closed bimodules having the same subset of rank one operators must coincide. Hence  $\J(\M)=\{T\in \B(\H)\colon \phi(P)^\perp  TP=0\}$, as required. The left order continuity of $\phi$ is a consequence of Lemma \ref{l_phi}.

 Assume that $\M$ is the  weakly closed subspace of $\B(\H)$ defined in (iv).  Observe that, for every $P, Q\in \N$, the subspace $Q\B(H) P^\perp$ is a $\T(\N)$-bimodule. Hence $\sspan\bigl(\{Q\B(H) P^\perp  \colon  P,Q\in \N, Q\B(H) P^\perp \subseteq \M\}\bigr)$ is also a $\T(\N)$-bimodule, and from this  immediately follows that the closure of this set in the weak operator topology is a  weakly closed $\T(\N)$-bimodule. Hence (iv)  $\Rightarrow$ (i). 
 
 On the other hand, by Proposition \ref{p_eqcond}, if (i) holds, then  $\M=\J(\M)$. By Lemma \ref{l_uem}, it follows that $
 \M= \fecho{\sspan}^w\bigl(\bigcup_{ x\otimes y\in \C(\M)}  P_y\B(\H) \hat{P}_x^\perp \bigr)$. Hence, by Lemma \ref{l_base}, $\M$ is contained in the weakly closed $\T(\N)$-bimodule 
 $$
 \fecho{\sspan}^w\bigl(\{Q\B(H) P^\perp  \colon  P,Q\in \N, Q\B(H) P^\perp \subseteq \M\}\bigr).
 $$
If $\M$ were properly contained in this set, then, by Corollary \ref{c_rankone}, there would  exist a rank one operator $x\otimes y=P_y(x\otimes y) \hat{P}_x$ lying in 
 $
  \fecho{\sspan}^w\bigl(\{Q\B(H) P^\perp  \colon  P,Q\in \N, Q\B(H) P^\perp \subseteq \M\}\bigr)
 $
 but not in $\M$. Consequently, since this latter subspace is a weakly closed bimodule, by Lemma \ref{l_base}, 
 $P_y\B(H) \hat{P}_x\subseteq \fecho{\sspan}^w\bigl(\{Q\B(H) P^\perp  \colon  P,Q\in \N, Q\B(H) P^\perp \subseteq \M\}\bigr)\subseteq \M$.
  Hence $x\otimes y$ would lie in $\M$, which contradicts the assumption.
\end{proof}

\begin{corollary} \label{c_largebim}
Let $\M$ be a weakly closed subspace of $\B(\H)$ and let $\J(\M)$ be as in \eqref{eq6}. 
Then   $\J(\M)$  is the largest  weakly closed  $\T(\N)$-bimodule contained in $\M$.
\end{corollary}

\begin{proof}  
It is clear that,  if  $\M$ is weakly closed, then  $\J(\M)=\fecho{\sspan}^w(\C(\M)) \subseteq   \M$. 

Suppose that there exists a weakly closed bimodule $\V$ contained in $\M$ which properly contains $\V$. Then, by Lemma \ref{l_eqcond}, $\V=\J(\V)$ and, thus, $\J(\V)= \fecho{\sspan}^w(\C(\V))\supsetneq \J(\M)$. Hence,  by Corollary \ref{c_rankone},  there exists a rank one operator $x\otimes y\in \V\backslash \J(\M)$ such that 
\begin{equation}\label{06}
P_y\B(\H) \hat{P}_x^\perp  \subseteq \M. 
\end{equation}
But this is impossible since, by the definition of $\J(\M)$, this bimodule contains all  operators $x\otimes y$ satisfying \eqref{06}. Hence $\J(\M)$ is the largest weak operator closed bimodule contained in $\M$.
\end{proof}

\section{Lie modules}\label{liemod}
 
The following  theorem  summarises the results of the previous section when $\M$ is a Lie $\T(\N)$-module.

\begin{theorem} \label{t_largebim} Let $\Ll$ be a weakly closed Lie $\T(\N)$-module of $\B(\H)$ and let 
$$
\C(\Ll) =\{x\otimes y\in \B(\H)\colon P_y\B(\H) \hat{P}_x^\perp  \subseteq \Ll\}.
$$
 Then, the largest   weakly closed $\T(\N)$-bimodule contained in $\Ll$ is $\J(\Ll) =\fecho{\sspan}^w (\C(\Ll))$. Moreover,
$$\J(\Ll) = \fecho{\sspan}^w\bigl(\bigcup_{ x\otimes y\in \C(\Ll)}  P_y\B(\H) \hat{P}_x^\perp \bigr)=\fecho{\sspan}^w\bigl(\{Q\B(\H) P^\perp  \colon  P,Q\in \N, Q\B(\H) P^\perp \subseteq \Ll\}\bigr).$$
\end{theorem}

\begin{lemma}\label{l_aux2}
Let $\Ll$ be a Lie $\T(\N)$-module and let $P,Q\in \T(\N)$ 
be mutually orthogonal projections. 
 Then, for all $T\in \Ll$, the operators $PTQ, QTP$ lie in $\Ll$.
\end{lemma}
\begin{proof}
Since $PQ=0$, it is easily seen that  $QTP=\frac{1}{2}([[[T,P],Q],Q]-[[T,P],Q])$, from which  follows that $QTP\in \Ll$. The remaining assertion can be similarly proved.
\end{proof}

\begin{lemma}\label{l_aux3} Let  $\Ll$ be a weakly closed Lie $\T(\N)$-module and let $P$ be a projection in $\N$. If  $P^\perp  \Ll P\neq \{0\}$, then $P\Ll P^\perp =P\B(\H)P^\perp $.
\end{lemma}

\begin{proof}
Let $P\in \N$ and $T\in\Ll$ be such that   $P^\perp  TP\neq 0$. Notice that Lemma \ref{l_aux2} guarantees that $P^\perp TP\in \Ll$. 
 To prove the assertion, it suffices to show that,  for all $x,y\in \H$,  the operator  $P(x\otimes y)P^\perp$  lies in $\Ll$. This trivially holds when $P(x\otimes y)P^\perp=0$. Assume now that $P(x\otimes y)P^\perp$ is  a rank one operator. Then 
\begin{equation}\label{eq5}
[[P(x\otimes y)P^\perp ,P^\perp  TP],P(x\otimes y)P^\perp ]=2P(x\otimes y)P^\perp  TP(x\otimes y)P^\perp
\end{equation}
and, therefore,
\begin{equation}\label{eq3}
[[P(x\otimes y)P^\perp ,P^\perp  TP],P(x\otimes y)P^\perp ]=2\langle P^\perp  TPy, x\rangle P (x\otimes y)P^\perp  
\end{equation}
lies in $\Ll$.
It  follows that $P(x\otimes y)P^\perp \in \Ll$, whenever   $\langle P^\perp  TPy, x\rangle\neq 0$. 

On the other hand, if $x\perp P^\perp  TPy$, then suppose firstly that $P^\perp TPy\neq 0$. 
 In this case, replacing $x\otimes y$ by $P^\perp  TPy \otimes Py$ in the above computations yields that  the operator $P^\perp  TPy \otimes Py$ lies in $\Ll$.
  Notice that the condition under which it  can be deduced from \eqref{eq3} that $P^\perp  TPy \otimes Py\in \Ll$ is, in this case, that $\langle P^\perp  TPy,P^\perp  TPy\rangle\neq 0$, which clearly holds. Moreover, since $\langle P^\perp  TPy-x,P^\perp  TPy\rangle\neq 0$, it also follows from \eqref{eq3} that $(P^\perp  TPy-P^\perp x)\otimes Py$ lies in $\Ll$. Hence, 
$$
P(x\otimes y)P^\perp=P^\perp  TPy\otimes Py - (P^\perp T Py- P^\perp x)\otimes P y
$$
 lies in $\Ll$.

Assume now that   $P^\perp TPy=0$.
 Since $P^\perp TP\neq 0$, there exists $z\in\H$ such that $P^\perp TPz\neq 0$, from which follows that 
 $P^\perp  TP(z-y)\neq 0$. Applying a reasoning similar to that of the preceding paragraph, it follows that both $ P(x\otimes z)P^\perp$ and $P(x\otimes (z-y))P^\perp  $ lie in $\Ll$. Hence, $P(x\otimes y) P^\perp=P(x\otimes z)P^\perp- P(x\otimes(z-y))P^\perp$ lies in $\Ll$, which concludes the proof.
\end{proof}

Let $\Ll$ be a Lie $\T(\N)$-module and let 
$\K(\Ll)$ be the subspace of $\B(\H)$ defined by 
\begin{equation}\label{05}
\K(\Ll)=\K_V(\Ll)+\K_L(\Ll)+\K_D(\Ll)+\K_\Delta(\Ll),
\end{equation}
 where
\begin{align}
\label{09}\K_V(\Ll)&=\overline{\spn}^w\{PTP^\perp \colon P\in \N,T\in\Ll\},\\
\label{10}\K_L(\Ll)&=\overline{\spn}^w\{P^\perp TP\colon P\in \N,T\in\Ll\},\\
\label{11}\K_D(\Ll)&=\overline{\spn}^w\{PSP^\perp TP\colon P\in \N,T\in\Ll,S\in \T(\N)\},\\
\label{12}\K_\Delta(\Ll)&=\overline{\spn}^w\{P^\perp TPSP^\perp \colon P\in \N,T\in \Ll,S\in \T(\N)\}.
\end{align}

\begin{lemma}\label{l_bimk}
Let $\Ll$ be a weakly closed
 Lie $\T(\N)$-module and let $\K(\Ll)$ and $\K_V(\Ll)$ be as in \eqref{05} and \eqref{09}, respectively. Then,     $\K(\Ll)$   is a weakly closed $\T(\N)$-bimodule and $\K_V(\Ll)$ is a weakly closed  ideal of $\T(\N)$  such that   $\K_V(\Ll)\subseteq \J(\Ll)$. 
\end{lemma}

\begin{remark}\label{rem5} Notice that  $\K_V(\Ll)$ is a  subspace of $\T(\N)$ and that, by Lemma \ref{l_aux2}, the spaces $\K_V(\Ll)$ and $\K_L(\Ll)$ are contained in $\Ll$. 
\end{remark}
\begin{proof} 
It is clear that $\K(\Ll)$ and $\K_V(\Ll)$ are weakly closed subspaces of $\B(\H)$ and, as observed in  Remark \ref{rem5},  $\K_V(\Ll)\subseteq \T(\N)$. 

To see that $\K_V(\Ll)$ is an  ideal of $\T(\N)$, it suffices to show that, for all $T\in \Ll, P\in \N, S\in \T(\N)$ one has that  both $PTP^\perp S$ and  $SPTP^\perp$ lie in $\K_V(\Ll)$. 
 Since $P^\perp SP^\perp \in \T(\N)$ and since, by Lemma \ref{l_aux2}, $PTP^\perp$ lies in $\Ll$, it follows that 
 $$
PTP^\perp S=PTP^\perp P^\perp SP^\perp =[PTP^\perp ,P^\perp SP^\perp]
$$
lies in  $\Ll$. But $PTP^\perp S=P(PTP^\perp S)P^\perp$, which shows that $PTP^\perp S$ lies in $\K_V(\Ll)$.

Similarly, 
$$SPTP^\perp=PSPPTP^\perp=[PSP,PTP^\perp]$$
lies in $\Ll$ and, therefore, $SPTP^\perp=P(SPTP^\perp)P^\perp$ lies in $\K_V(\Ll)$.
 Since it has been shown that $\K_V(\Ll)$ is a weakly closed ideal of $\T(\N)$, it immediately follows from Theorem \ref{t_largebim} that $\K_V(\Ll)\subseteq \J(\Ll)$.

It will be shown next that $\K_L(\Ll) \T(\N), \T(\N)\K_L(\Ll)\subseteq \K(\Ll)$. It suffices to show that, for all $T\in \Ll, P\in \N$ and $S\in \T(\N)$, the operators 
$P^\perp TP S, SP^\perp TP$ lie in $\K(\Ll)$. Observe also that, if $T$ is an operator in the Lie module $\Ll$, then, by Lemma \ref{l_aux2}, the operator $P^\perp TP$ lies  in  $\Ll$. 
Hence, it suffices to assume that $T\in \Ll$ is such that 
 $T=P^\perp TP$, for some $P\in \N$, and then prove that $TS, ST\in \K(\Ll)$, for all $S\in \T(\N)$.

 Let $T$ be an operator in $\Ll$ such that $T=P^\perp TP$, and let $S$ be an operator in the nest algebra. It follows that 
  $$
\begin{aligned}
TS& =P^\perp TPSP+P^\perp TPSP^\perp .
\end{aligned}$$
 It is clear that $P^\perp TPSP^\perp \in \K_\Delta(\Ll)$. On the other hand, 
 $$
\begin{aligned}
P^\perp TPSP
&=[P^\perp TP, PSP]\\
&=P^\perp [T,PSP]P.
\end{aligned}$$
Since $[T,PSP]\in \Ll$, it follows   that  $P^\perp TPSP\in \K_L(\Ll)$. Hence,  $TS$ 
lies in $\K(\Ll)$, as required.

Similarly,
$$
ST=P^\perp ST+PST=[P^\perp SP^\perp ,T]+PSP^\perp TP
$$
lies in $\K(\Ll)$, since $PSP^\perp TP\in \K_D(\Ll)$ and $[P^\perp SP^\perp ,T]=P^\perp [P^\perp SP^\perp ,T] P$ lies in  $\K_L(\Ll)$.

\medskip

To show  that $\K_D(\Ll) \T(\N), \T(\N)\K_D(\Ll)\subseteq \K(\Ll)$, it suffices to prove that for all $T\in \Ll, S, R\in \T(\N)$ and $P\in \N$, the operators  $PSP^\perp TPR$ and $RPSP^\perp TP$ lie in $\K(\Ll)$. 

As to the operator $RPSP^\perp TP$, observe that $RPSP^\perp TP=P(RPS)P^\perp TP$ and, since $RPS\in \T(\N)$, it immediately follows that $RPSP^\perp TP\in \K_D(\Ll)$. Hence, $\T(\N)\K_D(\Ll)\subseteq \K(\Ll)$.

It only remains to show that $PSP^\perp TPR\in \K(\Ll)$.  Observe that, by Lemma \ref{l_aux3}, either $P\Ll P^\perp=P\B(\H)P^\perp$ or $P^\perp \Ll P =\{0\}$. In the latter case, it is obvious that the assertion to be proved  trivially holds. 
 In the former case, notice that, by Lemma \ref{l_aux2}, 
 $P\B(\H)P^\perp \subseteq \Ll$. 
 
Let $T,S,R$ be as above and let $P\in \N$ be such that $P\B(\H)P^\perp \subseteq \Ll$. Then,
\begin{align*}
PSP^\perp TPR&=PSP^\perp TPRP+PSP^\perp TPRP^\perp \\
&=PSP^\perp [P^\perp TP,PRP]P+PSP^\perp TPRP^\perp.
\end{align*}
As seen above, $P\B(\H)P^\perp \subseteq \Ll$ yielding that the operator  
$
PSP^\perp TPRP^\perp
$
 lies in $\Ll$. Consequently, 
 $$
PSP^\perp TPRP^\perp=P(PSP^\perp TPRP^\perp)P^\perp
$$
 lies in  $\K_V(\Ll)$. Moreover, by Lemma \ref{l_aux2}, $P^\perp TP\in \Ll$, from which follows that $[P^\perp TP,PRP]\in \Ll$.
 Hence,  $PSP^\perp [P^\perp TP,PRP]P\in \K_D(\Ll)$. It follows that $\K_D(\Ll) \T(\N)\subseteq \K(\Ll)$.
 
 Finally, it will be shown that $\K_\Delta(\Ll) \T(\N), \T(\N)\K_\Delta(\Ll)\subseteq \K(\Ll)$. That is to say that, it must be proved that, for all $T\in \Ll, S, R\in \T(\N)$ and $P\in \N$, the operators $P^\perp TPSP^\perp R$ and $R P^\perp TPSP^\perp$ lie in $\K(\Ll)$.
 
 Suppose again that $P\Ll P^\perp =P\B(\H)P^\perp$. Recall that, by Lemma \ref{l_aux3}, the only other possibility is $P^\perp \Ll P =\{0\}$, in which case  the assertions to be proved trivially hold.
 
 Since  $SP^\perp R\in\T(\N)$, it follows that 
$P^\perp TPSP^\perp R=P^\perp TP(SP^\perp R)P^\perp$
lies in  $\K_\Delta(\Ll)$.
Furthermore,
\begin{align*}
RP^\perp TPSP^\perp &=PRP^\perp TPSP^\perp +P^\perp RP^\perp TPSP^\perp \\
&=PRP^\perp TPSP^\perp + P^\perp [P^\perp RP^\perp ,P^\perp TP]PSP^\perp.
\end{align*}
 Observe that  $PRP^\perp TPSP^\perp \in P\B(\H)P^\perp \subseteq \K_V(\Ll)$, since it is assumed that  $P\Ll P^\perp =P\B(\H)P^\perp$. Moreover,  $P^\perp [P^\perp RP^\perp ,P^\perp TP]PSP^\perp$ lies in  $\K_\Delta(\Ll)$, since $[P^\perp RP^\perp ,P^\perp TP]\in\Ll$.
\end{proof}

\begin{lemma}\label{l_comm}Let $\Ll$ be a weakly closed Lie $\T(\N)$-module and let $\K(\Ll)$  be the weakly closed $\T(\N)$-bimodule associated with $\Ll$   in  \eqref{05}. Then 
$[\K(\Ll),\T(\N)]\subseteq \Ll$.
\end{lemma}

\begin{proof} Since $\K_V(\Ll),\K_L(\Ll)\subseteq \Ll$, 
it is enough to prove that $[\K_D(\Ll),\T(\N)],[\K_\Delta(\Ll),\T(\N)]\subseteq \Ll$.  That is to say that, it suffices to show that for all $T\in \Ll, P\in \N$ and $R, S\in \T(\N)$, the operators $[PSP^\perp TP, R]$ and $[P^\perp TPSP^\perp, R]$ lie in $\Ll$.

Recall once again that,  given $P\in \N$, by Lemma \ref{l_aux3},   either $P\Ll P^\perp =P\B(\H)P^\perp$ or $P^\perp \Ll P =\{0\}$. In the latter case, for all $T\in \Ll$,  $P^\perp TP=0$, from which follows that the assertions to be proved are  trivially true. 

Suppose  now that $P\Ll P^\perp =P\B(\H)P^\perp$ and that $T\in \Ll$ is such that $P^\perp TP\neq 0$, in which case, by Lemma \ref{l_aux2}, $P^\perp TP\in\Ll$. Then, for all $R,S\in \N$, 
\begin{align*}
[PSP^\perp TP,R]&=[PSP^\perp TP,RP]+[PSP^\perp TP,PRP^\perp ]
+[PSP^\perp TP,P^\perp RP^\perp ]\\
&=[PSP^\perp TP-P^\perp TPSP^\perp ,RP]
+PSP^\perp TPRP^\perp \\
&=[[PSP^\perp ,P^\perp TP],RP]+PSP^\perp TPRP^\perp 
\end{align*}
lies in  $\Ll.$
 Similarly, 
\begin{align*}
[P^\perp TPSP^\perp ,R]&=[P^\perp TPSP^\perp ,RP]+[P^\perp TPSP^\perp ,PRP^\perp ]
+[P^\perp TPSP^\perp ,P^\perp RP^\perp ]\\
&=-PRP^\perp TPSP^\perp 
+[P^\perp TPSP^\perp -PSP^\perp TP,P^\perp RP^\perp ]\\
&=-PRP^\perp TPSP^\perp 
+[[P^\perp TP,PSP^\perp ],P^\perp RP^\perp ]
\end{align*}
is an operator in $\Ll$, which concludes the proof.
\end{proof}

Recall that it is possible to associate with  each   weakly closed $\T(\N)$-bimodule  $\K$    a (not necessarily unique) left order continuous homomorphism 
 $\phi\colon \N\to\N$ such that  
 $
\K=\{ T\in \B(\H)\colon \phi(P)^\perp TP=0\}
$
  (see Lemma \ref{l_eqcond} and Remark \ref{rem03}).

\begin{lemma}\label{l_aux4} Let $\Ll$ be a weakly closed Lie $\T(\N)$-module, let $\K(\Ll)$ be the weakly closed $\T(\N)$-bimodule defined  in \eqref{05}--\eqref{12},  and let $\phi(P)\colon \N\to\N$ be a left order continuous homomorphism associated with  
$
\K(\Ll).
$ 
 If
 $P\in \N$ is such that $\phi(P)<P$, then, for all $T\in \Ll$ and all $Q\in \N$ with $\phi(P)<Q<P$, $(Q-\phi(P))T(P-Q)=0$.
\end{lemma}

\begin{proof}
Let $T$ be an operator in $\Ll$ and let that $P, Q\in \N$. Since, by the definition \eqref{05}--\eqref{12} of \K(\Ll),  $QTQ^\perp\in \K(\Ll)$, it follows that $\phi(P)^\perp (QTQ^\perp)P=0$. Hence,  if $\phi(P)<Q<P$, then
$(Q-\phi(P))T(P-Q)=0$, as required.
\end{proof}

 Given  a weakly closed Lie $\T(\N)$-module $\K$, define $\D_{\K}$ as
 the algebra consisting of all operators $T\in \D(\N)$ such that, for every $P\in \N$ for which $\phi(P)<P_{-}$, there exists $\lambda_P$ in $\mathbb{C}$ satisfying the equality $T\bigl(P-\phi(P)\bigr)=\lambda_P \bigl(P-\phi(P)\bigr)$. The algebra $\D_{\K}$ is a von Neumann subalgebra of $\D(\N)$ and, when $\K$ is a weakly closed Lie ideal of $\T(\N)$, the  algebra  $\D_{\K}$ is exactly that defined in \cite{HMS}.

The next  lemma is  inspired by results of    \cite{HMS} and by the proofs therein.
\begin{lemma}\label{l_diag}Let $\Ll$ be a weakly closed Lie $\T(\N)$-module. Then 
$\Ll\subseteq \K(\Ll)+\D_{\K(\Ll)}$.
\end{lemma}

\begin{proof} 
Let $\pi$ be an expectation of $\T(\N)$ on $\D(\N)$ (see \cite{kd}, Corollary 8.5). Given $T\in \Ll$,  let $T= T_\pi +\pi(T)$, where
$T_\pi=T-\pi(T)$. Firstly, it will be shown that $T_\pi\in \K(\Ll)$; that is to say that, for all $P\in \N$, $\phi(P)^\perp T_\pi P=0$, where $\phi\colon \N\to\N$ is a left order continuous homomorphism on $\N$ associated with the  bimodule $\K(\Ll)$.

Let $Q$ be a projection in $\N$. Notice  that $\phi(P)^\perp Q^\perp TQP=0$, since  $Q^\perp TQ\in \K(\Ll)$ (see \eqref{05}--\eqref{12}). Then, 
\begin{align*}
Q^\perp (\phi(P)^\perp T_\pi P)Q&=\phi(P)^\perp (Q^\perp T_\pi Q)P\\
&=\phi(P)^\perp Q^\perp TQP-\phi(P)^\perp Q^\perp\pi(T)QP\\
&=-\phi(P)^\perp Q^\perp\pi(T)QP.
\end{align*}
But, since by \cite{kd}, Theorem 8.1,
$\phi(P)^\perp Q^\perp\pi(T)QP=\pi\bigl(\phi(P)^\perp Q^\perp TQP\bigr)$, it follows that, for all $Q\in \N$,  
$Q^\perp (\phi(P)^\perp TP)Q=0$. 
 Similarly,
\begin{align*}
Q\phi(P)^\perp T_\pi PQ^\perp&=\phi(P)^\perp (QTQ^\perp)P-\phi(P)^\perp Q\pi(T)Q^\perp P\\
&=-\pi\bigl(\phi(P)^\perp QTQ^\perp P\bigr)=0.
\end{align*}
Hence,  for all $P, Q\in \N$, 
\begin{equation}\label{eq10}
\phi(P)^\perp T_\pi P=Q\phi(P)^\perp T_\pi PQ + Q^\perp \phi(P)^\perp T_\pi PQ^\perp,
\end{equation}
from which follows that  $\phi(P)^\perp T_\pi P\in \D(\N)$. 
 Hence, by \cite{kd}, Theorem 8.1,
\begin{align*}
\phi(P)^\perp T_\pi P&=\pi(\phi(P)^\perp T_\pi P)\\
&=\phi(P)^\perp \pi(T_\pi) P\\
&=\phi(P)^\perp \pi(T-\pi(T)) P.
\end{align*} 
Since  $\pi(T-\pi(T))=0$, it follows that, for all $P\in \N$, $\phi(P)^\perp T_\pi P=0$ or, in other words, $T_\pi$ lies in $\K(\Ll)$.

It remains to show that $\pi(T)$ lies in $\D_{\K(\Ll)}$. Let $P\in \N$ be such that $\phi(P)< P_{-}$. Then, there exists a projection $Q\in \N$ such that $\phi(P)< Q< P$. 

Since $QTQ^\perp \in \K(\Ll)$ (see \eqref{05}--\eqref{12}), it follows that, for all $P\in \N$, $\phi(P)^\perp(QTQ^\perp)P=0$. Observe also that, since $T_{\pi}\in \K(\Ll)$, by Lemma \ref{l_comm},  $[\pi(T), \T(\N)]\subseteq \Ll$.  Hence, for all  $x,y\in \H$,  $[\pi(T),(Q-\phi(P))(x\otimes y)(P-Q)]\in \Ll$. It follows, by Lemma \ref{l_aux4}, that
$$
(Q-\phi(P))[\pi(T),(Q-\phi(P))(x\otimes y)(P-Q)](P-Q)=0
$$
and, consequently,
$$
((P-Q)x\otimes(Q-\phi(P))\pi(T)(Q-\phi(P))y)=
((P-Q)\pi(T)^*(P-Q)x\otimes (Q-\phi(P))y)
$$
Choosing $x,y\in \H$ such that $x=(P-Q)x$ and $y=(Q-\phi(P))y$, it is easy to see that    there must exist $\lambda_P\in\mathbb{C}$ such that
$$
\pi(T)(P-Q)=\lambda_P(P-Q)
$$
and 
$$\pi(T)\bigl(Q-\phi(P)\bigr)=\lambda_P\bigl(Q-\phi(P)\bigr).
$$
 It follows that 
$$
\pi(T)\bigl(P-\phi(P)\bigr)=\lambda_P\bigl(P-\phi(P)\bigr),
$$
 yielding that $\pi(T)$ lies in $\D_{\K(\Ll)}$, as required.
\end{proof}

Given a weakly closed Lie $\T(\N)$-module, let $\J(\Ll)$ be the  $\T(\N)$-bimodule defined at the beginning of this section and let $\K(\Ll)$ be defined by \eqref{05}--\eqref{12}. The next theorem  summarises the main results of Section \ref{liemod}. 

\begin{theorem}\label{t_sum}
Let $\Ll$ be a weakly closed Lie $\T(\N)$-module. Then, there exist weakly closed  $\T(\N)$-bimodules $\J(\Ll)$ and $\K(\Ll)$ and a von Neumman subalgebra $\D_{\K(\Ll)}$ of the diagonal $\D(\N)$ such that 
$$
\J(\Ll)\subseteq \Ll \subseteq  \K(\Ll) +\D_{\K(\Ll)}.
$$
 Moreover, 
 $$\J(\Ll) =\fecho{\sspan}^w\bigl(\{Q\B(\H) P^\perp  \colon  P,Q\in \N, Q\B(\H) P^\perp \subseteq \Ll\}\bigr)
 $$
  is the largest weakly closed  $\T(\N)$-bimodule contained in $\Ll$ and $\K(\Ll)$ is such that 
   $[\K(\Ll),\T(\N)]\subseteq \Ll$.
\end{theorem}

\begin{example}\label{example}
Notice  that neither is it necessarily the case   that $\J(\Ll)\subseteq \K(\Ll)$ nor  that $\Ll \subseteq  \K(\Ll)$. A simple counter-example can be given in the nest algebra of the $5\times 5$ upper triangular  complex matrices. Consider the Lie module $\Ll=\spn  \{I\} + \J(\Ll)$, where $\J(\Ll)$ is the bimodule consisting of the  $5\times 5$   complex  matrices such that 
$a_{i1}=0$, if $1\leq i\leq 5$, and $a_{i2}=0$, if $3\leq i\leq 5$. In this case, $\K(\Ll)$ consists of the matrices in $\J(\Ll)$ such that $a_{22}=0$.
\end{example}

\begin{remark}\label{finalrem}
When $\Ll$ is a weakly closed Lie ideal, $\K_L(\Ll), \K_D(\Ll), \K_\Delta(\Ll)=\{0\}$. In this situation, it has been shown in \cite{AO} that there exists a certain unital weakly closed  $*$-subalgebra $\breve{\mathcal{D}}(\Ll)$  of  $\D_{\K(\Ll)}$ such that
$$
\K(\Ll)\subseteq \J(\Ll)\subseteq \Ll \subseteq  \K(\Ll) +\D_{\K(\Ll)}=\J(\Ll) \oplus\breve{\mathcal{D}}(\Ll).
$$
\end{remark}

\end{document}